\documentclass[11pt,a4paper]{article}

\usepackage{amssymb,amsmath,amsthm}
\usepackage{braket,color}

\theoremstyle{plain}
\newtheorem{thm}{Theorem}[section]
\newtheorem{prop}[thm]{Proposition}
\newtheorem{lem}[thm]{Lemma}

\theoremstyle{definition}
\newtheorem{dfn}[thm]{Definition}
\newtheorem{rem}[thm]{Remark}

\numberwithin{equation}{section}

\makeatletter
\renewenvironment{proof}[1][\proofname]{\par
  \pushQED{\qed}%
  \normalfont \topsep6\p@\@plus6\p@\relax
  \trivlist
  \item[\hskip\labelsep
	\bfseries
    #1\@addpunct{.}]\ignorespaces
}{%
  \popQED\endtrivlist\@endpefalse
}
\makeatother

\DeclareMathOperator{\ev}{ev}

\DeclareMathOperator{\tr}{tr}
\DeclareMathOperator{\ad}{ad}

\newcommand{\ve}{\varepsilon}

\newcommand{\bbZ}{\mathbb{Z}}
\newcommand{\bbC}{\mathbb{C}}

\newcommand{\affY}{Y(\affsl)}

\newcommand{\gl}{\mathfrak{gl}}
\newcommand{\fraksl}{\mathfrak{sl}}
\newcommand{\affsl}{\hat{\mathfrak{sl}}_N}
\newcommand{\bfid}{\mathbf{1}}

\newcommand{\affnil}{\hat{\mathfrak{n}}}

\newcommand{\affCar}{\hat{\mathfrak{h}}}

\title{On Guay's evaluation map for affine Yangians}

\author{Ryosuke Kodera}

\date{}

\makeatletter
\let\@old@@maketitle=\@maketitle
\def\@maketitle{%
\footnotetext{%
\hspace*{-1em}\hspace*{-\footnotesep}%
E-mail address: kodera@math.s.chiba-u.ac.jp
}
\@old@@maketitle
}
\makeatother

\begin{document}
\maketitle

\begin{abstract}
We give a detailed proof of the existence of evaluation map for affine Yangians of type A to clarify that it needs an assumption on parameters.
This map was first found by Guay but a proof of its well-definedness and the assumption have not been written down in the literature.
We also determine the highest weights of evaluation modules defined as the pull-back of integrable highest weight modules of the affine Lie algebra $\hat{\gl}_N$ by the evaluation map.
\end{abstract}

\section{Introduction}

The affine Yangian $\affY$ of type A is a two-parameter deformation of the universal enveloping algebra of the universal central extension of a double loop Lie algebra $\mathfrak{sl}_N[s,t^{\pm 1}]$.
It may be regarded as an additive degeneration of the quantum toroidal algebra.
While the representation theory of quantum toroidal algebras and affine Yangians is expected to be very rich, it is still mysterious and intricate.
The author constructed certain representations of $\affY$ in \cite{MR3898327, MR3869424}.
These results are degenerate analogs of works by Varagnolo-Vasserot~\cite{MR1626481}, Saito-Takemura-Uglov~\cite{MR1603798}, Takemura-Uglov~\cite{MR1710750}, Nagao~\cite{MR2583334} in the quantum toroidal case.
Other known constructions of representations are by the Schur-Weyl type functor \cite{MR2199856}, by geometric approaches \cite{MR1818101,MR2827177}, and by vertex operators \cite{MR4014633}.
A different method to construct representations of the affine Yangian is considered in this paper.

Guay introduced an algebra homomorphism from the affine Yangian $\affY$ to a completion of $U(\hat{\gl}_N)$ in \cite{MR2323534}.
This is an affine analog of the well-known evaluation map from the Yangian $Y(\mathfrak{sl}_N)$ to $U(\gl_N)$.
We refer to Guay's map as evaluation map for the affine Yangian.
Since the classical evaluation map for $Y(\mathfrak{sl}_N)$ plays a fundamental role for the construction of irreducible representations, Guay's evaluation map should be important to develop the representation theory of the affine Yangian.

In \cite{MR2323534}, explicit values of the evaluation map are given only for part of the generators of $\affY$.
It is stated that the values for other generators are determined from the defining relations of $\affY$ and a property of the map, and a proof of its well-definedness is omitted.
One of the goals of the present paper is to provide a detailed proof of the existence of the evaluation map for $\affY$.
To accomplish this, we give explicit values of the evaluation map for distinguished generators of degree one.
By results of Guay~\cite{MR2323534} and Guay-Nakajima-Wendlandt~\cite{MR3861718}, the affine Yangian is generated by these and degree zero elements with reasonable defining relations.

In the course of checking the well-definedness, we have realized that we need to impose a certain relation among the parameters and the central element of the affine Yangian.
To clarify this fact is also a purpose to write this paper.
So far we have not yet been able to define the evaluation map for general parameters. 
The main results of the paper are Theorem~\ref{thm:Guay} and Theorem~\ref{thm:main}.
Theorem~\ref{thm:Guay} is Guay's original version.
The original Guay's evaluation map matches with lowest weight modules of $\hat{\gl}_N$.
To deal with highest weight modules, we introduce an opposite evaluation map in Theorem~\ref{thm:main}.

We can pull back integrable highest weight modules of $\hat{\gl}_N$ via the evaluation map in Theorem~\ref{thm:main} to make them $\affY$-modules.
The resulting modules satisfy the highest weight condition for affine Yangian.
We determine the highest weights of these evaluation modules in Theorem~\ref{thm:highest_weight} and investigate an analog of the Drinfeld polynomials.
The level-one Fock representation constructed by the author in \cite{MR3898327, MR3869424} turns out to be isomorphic to an evaluation module if we specialize the parameters.

We note that we consider the affine Yangian $\affY$ for $N \geq 3$.
Although we expect the existence of the evaluation map for $N=2$ case, we have not proved it.
The main reason is absence of Theorem~\ref{thm:GNW} which reduces the defining relations of the affine Yangian to those among the generators of degree zero and one.
The case of $Y(\hat{\mathfrak{gl}}_1)$ appears in a work by Schiffmann-Vasserot~\cite{MR3150250}.

The quantum toroidal case is studied by Miki~\cite{MR1694256} and Feigin-Jimbo-Mukhin~\cite{FJM}.
However, the evaluation map we consider here is not a direct analog of theirs.
Our evaluation map has a formula only for the generators of degree zero and one, while it is given for all generators of the quantum toroidal algebra in \cite{MR1694256, FJM}.
One can twist the homomorphism in \cite{MR1694256, FJM} by Miki's automorphism and obtain another evaluation map.
It seems that our evaluation map may be regarded as a degeneration of the latter one.

This paper is organized as follows.
Section~2 is devoted to preliminaries on the affine Yangian $\affY$ and the affine Lie algebra $\hat{\gl}_N$.
We provide completions of $U(\hat{\gl}_N)$ which are used to formulate the main results.
In Section~3, we give a proof of the existence of the two kinds of evaluation map.
Then we determine the highest weights of evaluation modules in Section~4.

\subsection*{Added remark}

After this paper was published online, an error was found.
There was an error in the proof of Theorem~\ref{thm:Guay} (main theorem) in the earlier version, and consequently it was wrong as stated.
We need to correct the definition of the affine Lie algebra $\hat{\gl}_N$.
More precisely we need to modify the defining relation of the diagonal Heisenberg part of $\hat{\gl}_N$.
Then we can show that there exists an algebra homomorphism from the affine Yangian to a completion of the universal enveloping algebra of $\hat{\gl}_N$ as desired.
The condition among parameters and the explicit form of the evaluation map need not to be changed.
The actual change in the proof of Theorem~\ref{thm:Guay} concerns the relation (\ref{eq:mainthmHH2}), namely Lemma~\ref{lem:HH2}.
Let us summarize corrections made in this version:
the definition of $\hat{\gl}_N$ in Section~\ref{subsection:affine_Lie_algebra}; the values of $[A_i,A_j]$ and $[B_i,A_j]$ in Lemma~\ref{lem:HH2} along with its proof; Remark~\ref{rem:thm}.

\subsection*{Acknowledgments}
The author thanks to Mamoru Ueda for pointing out an error in the published version of this paper.
This work was supported by JSPS KAKENHI Grant Number 17H06127 and 18K13390.

\section{Preliminaries}\label{sec:affine_Yangian}

\subsection{Affine Yangian}

Fix an integer $N \geq 3$ throughout the paper.
We use the notation $\{x,y\}=xy+yx$.

\begin{dfn}
The affine Yangian $\affY = Y_{\ve_1,\ve_2}(\affsl)$ is the algebra over $\bbC$ generated by $x_{i,r}^{+}, x_{i,r}^{-}, h_{i,r}$ $(i \in \mathbb{Z} / N\mathbb{Z}, r \in \mathbb{Z}_{\geq 0})$ with parameters $\ve_1, \ve_2 \in \bbC$ subject to the relations:
\begin{equation*}
	[h_{i,r}, h_{j,s}] = 0, 
	\quad [x_{i,r}^{+}, x_{j,s}^{-}] = \delta_{ij} h_{i, r+s}, 
	\quad [h_{i,0}, x_{j,r}^{\pm}] = \pm a_{ij} x_{j,r}^{\pm},
\end{equation*}
\begin{equation*}
	[h_{i, r+1}, x_{j, s}^{\pm}] - [h_{i, r}, x_{j, s+1}^{\pm}] 
	= \pm a_{ij} \dfrac{\varepsilon_1 + \varepsilon_2}{2} \{h_{i, r}, x_{j, s}^{\pm}\} 
	- m_{ij} \dfrac{\varepsilon_1 - \varepsilon_2}{2} [h_{i, r}, x_{j, s}^{\pm}],
\end{equation*}
\begin{equation*}
	[x_{i, r+1}^{\pm}, x_{j, s}^{\pm}] - [x_{i, r}^{\pm}, x_{j, s+1}^{\pm}] 
	= \pm a_{ij}\dfrac{\varepsilon_1 + \varepsilon_2}{2} \{x_{i, r}^{\pm}, x_{j, s}^{\pm}\} 
	- m_{ij} \dfrac{\varepsilon_1 - \varepsilon_2}{2} [x_{i, r}^{\pm}, x_{j, s}^{\pm}],
\end{equation*}
\begin{equation*}
	\sum_{w \in \mathfrak{S}_{1 - a_{ij}}}[x_{i,r_{w(1)}}^{\pm}, [x_{i,r_{w(2)}}^{\pm}, \dots, [x_{i,r_{w(1 - a_{ij})}}^{\pm}, x_{j,s}^{\pm}]\dots]] = 0 \ \ (i \neq j),
\end{equation*}
where
\[
	a_{ij} =
	\begin{cases}
		2  &\text{if } i=j, \\
		-1 &\text{if } i=j \pm 1, \\
		0  &\text{otherwise,}
	\end{cases}\quad
	m_{ij} =
	\begin{cases}
		1  &\text{if } j=i-1, \\
		-1 &\text{if } j=i+1, \\
		0  &\text{otherwise.}
	\end{cases}
\]
\end{dfn}
The subalgebra generated by $x_{i,0}^+, x_{i,0}^-, h_{i,0}$ ($i \in \bbZ / N\bbZ$) is isomorphic to $U(\affsl)$ (see \cite[Theorem~6.1]{MR2323534} for $N \geq 4$ and \cite[Theorem~6.9]{MR4014633} in general).
We set $\hbar = \ve_1 + \ve_2$ and $\tilde{h}_{i,1} = h_{i,1} - \dfrac{\hbar}{2} h_{i,0}^2$.

\begin{thm}[Guay~\cite{MR2323534}, Proposition~2.1, Guay-Nakajima-Wendlandt~\cite{MR3861718}, Theorem~2.12 and Section~6]\label{thm:GNW}
The affine Yangian $\affY$ is isomorphic to the algebra generated by $x_{i,r}^{+}, x_{i,r}^{-}, h_{i,r}$ $(i \in \mathbb{Z} / N\mathbb{Z}, r = 0,1)$ subject to the relations:
\begin{equation*}
[h_{i,r}, h_{j,s}] = 0,
\end{equation*}
\begin{equation*}
[x_{i,0}^{+}, x_{j,0}^{-}] = \delta_{ij} h_{i, 0},
 \quad [x_{i,1}^{+}, x_{j,0}^{-}] = \delta_{ij} h_{i, 1} = [x_{i,0}^{+}, x_{j,1}^{-}],
\end{equation*}
\begin{equation*}
\ [h_{i,0}, x_{j,r}^{\pm}] = \pm a_{ij} x_{j,r}^{\pm},
 \quad [\tilde{h}_{i,1}, x_{j,0}^{\pm}] = \pm a_{ij}\left(x_{j,1}^{\pm}-m_{ij}\dfrac{\varepsilon_1 - \varepsilon_2}{2} x_{j, 0}^{\pm}\right),
\end{equation*}
\begin{equation*}
[x_{i, 1}^{\pm}, x_{j, 0}^{\pm}] - [x_{i, 0}^{\pm}, x_{j, 1}^{\pm}] = \pm a_{ij}\dfrac{\varepsilon_1 + \varepsilon_2}{2} \{x_{i, 0}^{\pm}, x_{j, 0}^{\pm}\} - m_{ij} \dfrac{\varepsilon_1 - \varepsilon_2}{2} [x_{i, 0}^{\pm}, x_{j, 0}^{\pm}],
\end{equation*}
\begin{equation*}
(\ad x_{i,0}^{\pm})^{1-a_{ij}} (x_{j,0}^{\pm})= 0 \ \ (i \neq j), \label{eq:16}
\end{equation*}
where we set $\tilde{h}_{i,1} = {h}_{i,1} - \dfrac{\ve_1 + \ve_2}{2} h_{i,0}^2$.
\end{thm}

We can slightly reduce the relations as follows. 
Proofs are straightforward.
\begin{lem}\label{lem:reduce1}
The relation $[h_{i,0}, x_{j,1}^{\pm}]=\pm a_{ij} x_{j,1}^{\pm}$ is deduced from the following:
\[
	[\tilde{h}_{i,1}, h_{j,0}] = 0,\quad [h_{i,0}, x_{j,0}^{\pm}] = \pm a_{ij} x_{j,0}^{\pm},\quad [\tilde{h}_{i,1}, x_{i,0}^{\pm}]= \pm 2 x_{i,1}^{\pm}.
\]
\end{lem}
\begin{lem}\label{lem:reduce2}
The relation $[x_{i,0}^{+}, x_{j,1}^{-}]=\delta_{ij} h_{i, 1}$ is deduced from the following:
\[
	[\tilde{h}_{i,1}, h_{j,0}] = 0,\quad [x_{i,r}^{+}, x_{j,0}^{-}]=\delta_{ij} h_{i, r} \ \ \text{for $r=0,1$},
\]
\[ [\tilde{h}_{i,1}, x_{j,0}^{+}]= a_{ij}\left(x_{j,1}^{+}-m_{ij}\dfrac{\varepsilon_1 - \varepsilon_2}{2} x_{j, 0}^{+}\right),\quad	[\tilde{h}_{i,1}, x_{i,0}^{-}] = -2 x_{i,1}^{-}.
\]
\end{lem}

For each $\alpha \in \bbC$, we define an algebra automorphism $\tau_{\alpha}$ of $\affY$ by
\[
	x_{i,r}^{\pm} \mapsto \sum_{s=0}^r \dbinom{r}{s} \alpha^{r-s} x_{i,s}^{\pm}, \quad h_{i,r} \mapsto \sum_{s=0}^r \dbinom{r}{s} \alpha^{r-s} h_{i,s}.
\]
We can verify that $\tau_{\alpha}$ is well-defined in the same way as an automorphism $\rho$ given in \cite[Lemma~3.5]{MR2199856}.
The definition of $\rho$ will be recalled in \ref{subsection:cyclic}.
Let $\mu$ be the algebra anti-isomorphism from $Y_{\ve_1,\ve_2}(\affsl)$ to $Y_{-\ve_2,-\ve_1}(\affsl)$ defined by
\[
	x_{i,r}^{\pm} \mapsto - x_{i,r}^{\pm}, \quad h_{i,r} \mapsto - h_{i,r}.
\]
It is easy to see that the assignment respects the defining relations.

\subsection{Affine Lie algebra $\hat{\gl}_N$}\label{subsection:affine_Lie_algebra}

Let $\gl_N$ be the complex general linear Lie algebra consisting of $N \times N$ matrices.
We denote by $E_{i,j}$ the matrix unit with $(i,j)$-th entry $1$.
The indices $i,j$ of $E_{i,j}$ are regarded as elements of $\bbZ/N\bbZ$.
The transpose of an element $X$ of $\gl_N$ is denoted by ${}^t X$.
Put $\bfid = \sum_{i=1}^N E_{i,i}$ and $\mathfrak{a} = \bbC \bfid$.
We have a decomposition
\[
	\gl_N = \fraksl_N \oplus \mathfrak{a}.
\]

Let $\hat{\fraksl}_N = \fraksl_N \otimes \bbC[t,t^{-1}] \oplus \bbC c$ be the affine Lie algebra whose Lie bracket is given by
\[
	[X \otimes t^{r}, Y \otimes t^{s}] = [X,Y] \otimes t^{r+s} + r \delta_{r+s,0} \tr(XY) c, \quad \text{$c$ is central}.
\]
Let $\hat{\mathfrak{a}} = \mathfrak{a} \otimes \bbC[t,t^{-1}] \oplus \bbC c'$ be the Heisenberg Lie algebra whose Lie bracket is given by
\[
	[\bfid \otimes t^{r}, \bfid \otimes t^{s}] = r \delta_{r+s,0} N c', \quad \text{$c'$ is central}.
\]
Define two kinds of the affine Lie algebra $\hat{\gl}_N$ by
\[
	\hat{\gl}_N^{(+)} = \Big( \hat{\fraksl}_N \oplus \hat{\mathfrak{a}} \Big) / (c'- (c+N)), \quad \hat{\gl}_N^{(-)} = \Big( \hat{\fraksl}_N \oplus \hat{\mathfrak{a}} \Big) / (c'- (c-N)).
\]
Then we have
\[
	[E_{i,i} \otimes t^{r}, E_{j,j} \otimes t^{s}] = r \delta_{r+s,0} (\delta_{i,j}c + 1) \quad \text {in} \quad \hat{\gl}_N^{(+)} 
\]
and
\[
	[E_{i,i} \otimes t^{r}, E_{j,j} \otimes t^{s}] = r \delta_{r+s,0} (\delta_{i,j}c - 1) \quad \text {in} \quad \hat{\gl}_N^{(-)}.
\]
In the sequel, the symbol $\hat{\gl}_N$ denotes both $\hat{\gl}_N^{(+)}$ and $\hat{\gl}_N^{(-)}$ unless otherwise stated.

We denote the element $X \otimes t^s$ by $X(s)$.
We set
\begin{gather*}
	x_0^+ = E_{N,1}(1), \quad x_0^{-} = E_{1,N}(-1), \quad h_0 = E_{N,N} - E_{1,1} + c,\\
	x_{i}^{+} = E_{i,i+1}, \quad x_i^{-} = E_{i+1,i}, \quad h_i = E_{i,i}-E_{i+1,i+1} \ \ (i \neq 0).
\end{gather*}
Let $\affnil_{\pm}$ be the Lie subalgebras of $\hat{\gl}_N$ generated by $x_{i}^{\pm}$ ($i \in \bbZ/N\bbZ$) and $\bfid(s)$ ($\pm s > 0$).
That is,
\[
	\affnil_{+} = \bigoplus_{\substack{i < j\\ s \geq 0}} \bbC E_{i,j}(s) \oplus \bigoplus_{\substack{i \geq j\\ s > 0}} \bbC E_{i,j}(s), \quad \affnil_{-} = \bigoplus_{\substack{i > j\\ s \leq 0}} \bbC E_{i,j}(s) \oplus \bigoplus_{\substack{i \leq j\\ s < 0}} \bbC E_{i,j}(s).
\]
Let $\affCar$ be the Cartan subalgebra generated by $h_i$ ($i \in \bbZ/N\bbZ$) and $\bfid$.

Let $\omega_U$ be the algebra anti-automorphism of ${U}(\hat{\gl}_N)$ defined by $\omega_U(X(s))= {}^{t}X(-s)$ and $\omega_U(c)=c$.
We denote by $\mu_U$ the algebra anti-isomorphism from ${U}(\hat{\gl}_N^{(+)})$ to ${U}(\hat{\gl}_N^{(-)})$ induced from the assignment $X$ in $\hat{\gl}_N^{(+)}$ $\mapsto$ $-X$ in $\hat{\gl}_N^{(-)}$.
The restriction of $\mu_U$ to $\affsl$ gives an algebra anti-automorphism of $\affsl$.
The anti-isomorphism $\mu$ defined in the previous subsection is an extension of the restriction of $\mu_U$ to $\affsl$.

We define gradings of $\affnil_{\pm}$ by $\deg X(s) = s$. 
Then $U(\affnil_{\pm})$ become graded algebras.
We denote by $U(\affnil_{\pm})[s]$ the degree $s$ components.
Let us introduce completions of $U(\hat{\gl}_N^{(+)})$ and $U(\hat{\gl}_N^{(-)})$.
\begin{dfn}
We define completions $U(\hat{\gl}_N)_{{\rm comp},+}$ and $U(\hat{\gl}_N)_{{\rm comp},-}$ of $U(\hat{\gl}_N^{(+)})$ and $U(\hat{\gl}_N^{(-)})$, respectively, as follows:
\[
	U(\hat{\gl}_N)_{{\rm comp},+} = \bigoplus_{k \in \bbZ}\, \prod_{\substack{r, s \geq 0\\ s-r = k}} \left( U(\affnil_-)[-r] \otimes U(\affCar) \otimes U(\affnil_+)[s] \right),
\]
\[
	U(\hat{\gl}_N)_{{\rm comp},-} = \bigoplus_{k \in \bbZ}\, \prod_{\substack{r, s \geq 0\\ r-s = k}} \left( U(\affnil_+)[r] \otimes U(\affCar) \otimes U(\affnil_-)[-s] \right).
\]
\end{dfn}
Both $U(\hat{\gl}_N)_{{\rm comp},+}$ and $U(\hat{\gl}_N)_{{\rm comp},-}$ have natural algebra structures which contain $U(\hat{\gl}_N^{(+)})$ and $U(\hat{\gl}_N^{(-)})$ as subalgebras, respectively.
Moreover the anti-automorphism $\omega_U$ of $U(\hat{\gl}_N)$ extends to the completions and the anti-isomorphism
$\mu_U \colon {U}(\hat{\gl}_N^{(+)}) \to {U}(\hat{\gl}_N^{(-)})$
extends to an algebra anti-isomorphism
\[
	\mu_U \colon U(\hat{\gl}_N)_{{\rm comp},+} \to U(\hat{\gl}_N)_{{\rm comp},-}.
\]

\section{Evaluation map}
\subsection{Main theorem}
From now on, we evaluate the central element $c$ at a complex number, which is denoted also by the same letter $c$.

\begin{thm}[Guay~\cite{MR2323534}]\label{thm:Guay}
Assume $\hbar c = N \ve_2$.
Then there exists an algebra homomorphism $\ev \colon \affY \to U(\hat{\gl}_N)_{{\rm comp},-}$ uniquely determined by 
\begin{gather*}
	\ev(x_{i,0}^{+}) = x_{i}^{+}, \quad \ev(x_{i,0}^{-}) = x_{i}^{-},\quad \ev(h_{i,0}) = h_{i},
\end{gather*}
\begin{gather*}
	\ev(x_{i,1}^{+}) = \begin{cases}
		(1+ N \varepsilon_2) x_{0}^{+} + \hbar \displaystyle\sum_{s \geq 0} \sum_{k=1}^N E_{k,1}(s+1) E_{N,k}(-s) \text{ if $i = 0$},\\
		(1+ i \varepsilon_2) x_{i}^{+} + \hbar  \displaystyle\sum_{s \geq 0} \Big( \sum_{k=1}^i E_{k,i+1}(s) E_{i,k}(-s) + \sum_{k=i+1}^N E_{k,i+1}(s+1) E_{i,k}(-s-1) \Big) \\
		\qquad\qquad\qquad\qquad\qquad\qquad\qquad\qquad\qquad\qquad\qquad\qquad\qquad\qquad\text{ if $i \neq 0$},
	\end{cases}
\end{gather*}
\begin{gather*}
	\ev(x_{i,1}^-) = \begin{cases}
		(1+ N \varepsilon_2) x_{0}^{-} + \hbar \displaystyle\sum_{s \geq 0} \sum_{k=1}^N E_{k,N}(s) E_{1,k}(-s-1) \text{ if $i = 0$},\\
		(1+ i \varepsilon_2) x_{i}^{-} + \hbar \displaystyle\sum_{s \geq 0} \Big( \sum_{k=1}^i E_{k,i}(s) E_{i+1,k}(-s) + \sum_{k=i+1}^N E_{k,i}(s+1) E_{i+1,k}(-s-1) \Big) \\
		\qquad\qquad\qquad\qquad\qquad\qquad\qquad\qquad\qquad\qquad\qquad\qquad\qquad\qquad\text{ if $i \neq 0$},
	\end{cases}
\end{gather*}
\begin{gather*}
	\ev(h_{i,1}) = \begin{cases}
		(1+ N \varepsilon_2) h_{0} - \hbar E_{N,N} (E_{1,1}-c) \\
		\quad + \hbar \displaystyle\sum_{s \geq 0} \sum_{k=1}^{N} \Big( E_{k,N}(s) E_{N,k}(-s) - E_{k,1}(s+1) E_{1,k}(-s-1) \Big) \text{ if $i = 0$},\\
		\\
		(1+ i \varepsilon_2) h_{i} - \hbar E_{i,i}E_{i+1,i+1} \\
		\quad+ \hbar \displaystyle\sum_{s \geq 0} \Big( \sum_{k=1}^{i} E_{k,i}(s) E_{i,k}(-s) + \displaystyle\sum_{k=i+1}^{N}  E_{k,i}(s+1) E_{i,k}(-s-1) \\
		\qquad\qquad\quad - \displaystyle\sum_{k=1}^{i}E_{k,i+1}(s) E_{i+1,k}(-s) - \displaystyle\sum_{k=i+1}^{N} E_{k,i+1}(s+1) E_{i+1,k}(-s-1) \Big)\\
		\qquad\qquad\qquad\qquad\qquad\qquad\qquad\qquad\qquad\qquad\qquad\qquad\qquad\qquad \text{ if $i \neq 0$}.
	\end{cases}
\end{gather*}
\end{thm}
The formulas in the theorem are deduced from those for $\ev(H_{i,1})$ $(i \neq 0)$ where $H_{i,1}=h_{i,1} + (i/2)(\ve_1-\ve_2)h_{i,0}$, given in \cite[Section~6, pp.\ 462--463]{MR2323534}.
In \cite{MR2323534}, computations for the well-definedness are omitted.
Moreover we need the condition $\hbar c=N\ve_2$ to verify that those formulas preserve the defining relations of $\affY$ and this condition is not mentioned in the paper.
By these reasons, we give a detailed proof of the existence of the evaluation map $\ev$ as one of the main results of the present paper.  

\begin{rem}
It is straightforward to deduce the following formula for $\ev(\tilde{h}_{i,1})$: 
\begin{equation*}
	\begin{split}
	\ev(\tilde{h}_{0,1}) &= (1+ N \varepsilon_2) h_{0} - \dfrac{\hbar}{2} (E_{N,N}^2 +  (E_{1,1}-c)^2) \\
		&\qquad+ \hbar \sum_{s \geq 0} \sum_{k=1}^{N} \Big( E_{k,N}(s) E_{N,k}(-s) - E_{k,1}(s+1) E_{1,k}(-s-1) \Big),
	\end{split}
\end{equation*}
\begin{equation*}
	\begin{split}
	&\ev(\tilde{h}_{i,1}) = (1+ i \varepsilon_2) h_{i} - \dfrac{\hbar}{2} (E_{i,i}^2+E_{i+1,i+1}^2) \\
		&\qquad+ \hbar \sum_{s \geq 0} \Big( \sum_{k=1}^{i} E_{k,i}(s) E_{i,k}(-s) + \displaystyle\sum_{k=i+1}^{N}  E_{k,i}(s+1) E_{i,k}(-s-1) \\
		&\qquad\qquad\qquad - \displaystyle\sum_{k=1}^{i}E_{k,i+1}(s) E_{i+1,k}(-s) - \displaystyle\sum_{k=i+1}^{N} E_{k,i+1}(s+1) E_{i+1,k}(-s-1) \Big)\ \ (i \neq 0).
	\end{split}
\end{equation*}
\end{rem}

\begin{proof}
By the definition of $\ev$, the elements $\ev(x_{i,0}^{\pm})$, $\ev(h_{i,0})$ ($i \in \bbZ/N\bbZ$) automatically satisfy the defining relations of $\affsl$. 
By Theorem~\ref{thm:GNW}, Lemma~\ref{lem:reduce1} and \ref{lem:reduce2}, it is enough to show the following:
\begin{equation}
[\ev(h_{i,1}), h_{j}] = 0, \label{eq:mainthmHH1}
\end{equation}
\begin{equation}
[\ev(h_{i,1}), \ev(h_{j,1})] = 0, \label{eq:mainthmHH2}
\end{equation}
\begin{equation}
[\ev(x_{i,1}^{+}), x_{j}^{-}] = \delta_{ij} \ev(h_{i, 1}), \label{eq:mainthm+-}
\end{equation}
\begin{equation}
[\ev(\tilde{h}_{i,1}), x_{j}^{\pm}] = \pm a_{ij} \left( \ev(x_{j,1}^{\pm})-m_{ij}\dfrac{\varepsilon_1 - \varepsilon_2}{2} x_{j}^{\pm} \right), \label{eq:mainthmHX}
\end{equation}
\begin{equation}
[\ev(x_{i, 1}^{\pm}), x_{j}^{\pm}] - [x_{i}^{\pm}, \ev(x_{j, 1}^{\pm})] = \pm a_{ij}\dfrac{\hbar}{2} \{x_{i}^{\pm}, x_{j}^{\pm}\} - m_{ij} \dfrac{\varepsilon_1 - \varepsilon_2}{2} [x_{i}^{\pm}, x_{j}^{\pm}]. \label{eq:mainthmXX2}
\end{equation}
Moreover the relations (\ref{eq:mainthmHX}), (\ref{eq:mainthmXX2}) for $-$ are deduced from those for $+$ by applying the anti-automorphism $\omega_U$ since we have $\omega_U(\ev(h_{i,1}))=\ev(h_{i,1})$ and $\omega_U(\ev(x_{i,1}^+))=\ev(x_{i,1}^{-})$.

Let us start to check the relations.
We use the symbol $\delta(P)$ for $1$ if $P$ is true, $0$ otherwise. 
The relation~(\ref{eq:mainthmHH1}) clearly holds since $\ev(h_{i,1})$ has only weight $0$ terms.

We show (\ref{eq:mainthmHH2}).
We may assume $i < j$.
Further assume $i \neq 0$.
The proof for the case $i=0$ is similar.
Put
\begin{gather*}
	A_i = \sum_{r \geq 0} \sum_{k=1}^i E_{k,i}(r) E_{i,k}(-r), \quad	B_i = \sum_{r \geq 0} \sum_{k=i+1}^N E_{k,i}(r+1) E_{i,k}(-r-1), \\
	C_i = \sum_{r \geq 0} \sum_{k=1}^i E_{k,i+1}(r) E_{i+1,k}(-r), \quad D_i = \sum_{r \geq 0} \sum_{k=i+1}^N E_{k,i+1}(r+1) E_{i+1,k}(-r-1).
\end{gather*}
Since we have $[\ev(h_{i,1}),E_{k,k}]=0$ for all $k$, it is enough to show
\[
	[A_i+B_i-C_i-D_i,\, A_j+B_j-C_j-D_j] = 0.
\]
We give a proof for
\begin{equation}
	[A_i,A_j] + [A_i,B_j] + [B_i,A_j] + [B_i,B_j] = 0. \label{eq:vanish}
\end{equation}
We can similarly prove
\begin{gather*}
	[A_i,C_j] + [A_i,D_j] + [B_i,C_j] + [B_i,D_j] = 0,\\
	[C_i,A_j] + [C_i,B_j] + [D_i,A_j] + [D_i,B_j] = 0,\\
	[C_i,C_j] + [C_i,D_j] + [D_i,C_j] + [D_i,D_j] = 0.
\end{gather*} 
The identity (\ref{eq:vanish}) follows from next lemma.

\begin{lem}\label{lem:HH2}
We have
\begin{equation*}
	\begin{split}
		[A_i,A_j] &= \displaystyle\sum_{\substack{r,s \geq 0\\r>s}} \sum_{k=1}^i \Big( E_{k,i}(r)E_{i,j}(s-r)E_{j,k}(-s) - E_{k,j}(s)E_{j,i}(r-s)E_{i,k}(-r) \Big)\\
		&\quad + \sum_{r > 0} r \Big( E_{i,i}(r)E_{j,j}(-r) - E_{j,j}(r)E_{i,i}(-r) \Big),
	\end{split}
\end{equation*}
\[
		[A_i,B_j] = 0,
\]
\begin{equation*}
	\begin{split}
		&[B_i,A_j] \\
		&= \sum_{r,s \geq 0} \Bigg( \sum_{k=1}^i \Big(\! -E_{k,i}(r+s+1)E_{i,j}(-r-1)E_{j,k}(-s) + E_{k,j}(s)E_{j,i}(r+1)E_{i,k}(-r-s-1) \Big) \\
		&\qquad\qquad + \sum_{k=j+1}^N \Big(\! -E_{k,i}(r+1)E_{i,j}(s)E_{j,k}(-r-s-1) + E_{k,j}(r+s+1)E_{j,i}(-s)E_{i,k}(-r-1) \Big) \Bigg)\\
		&\quad + \sum_{r > 0} r \Big(\! -E_{i,i}(r)E_{j,j}(-r) + E_{j,j}(r)E_{i,i}(-r) \Big),
	\end{split}
\end{equation*}
\begin{equation*}
	\begin{split}
		&[B_i,B_j] \\
		&= \displaystyle\sum_{\substack{r,s \geq 0\\r \leq s}} \sum_{k=j+1}^N \Big( E_{k,i}(r+1)E_{i,j}(s-r)E_{j,k}(-s-1) - E_{k,j}(s+1)E_{j,i}(r-s)E_{i,k}(-r-1) \Big).
	\end{split}
\end{equation*} 
\end{lem}
\begin{proof}
We compute $[A_i,A_j]$ as
\begin{equation*}
	\begin{split}
		&[A_i,A_j] = \sum_{r,s \geq 0} \sum_{\substack{1 \leq k \leq i\\ 1 \leq l \leq j}}[E_{k,i}(r) E_{i,k}(-r),E_{l,j}(s) E_{j,l}(-s)]\\
		&= \sum_{r,s \geq 0} \sum_{\substack{1 \leq k \leq i\\ 1 \leq l \leq j}} \Big( [E_{k,i}(r),E_{l,j}(s)]E_{i,k}(-r)E_{j,l}(-s) + E_{k,i}(r) [E_{i,k}(-r),E_{l,j}(s)]E_{j,l}(-s)\\  
		&\qquad\qquad\qquad\quad + E_{l,j}(s)[E_{k,i}(r),E_{j,l}(-s)]E_{i,k}(-r) + E_{l,j}(s)E_{k,i}(r)[E_{i,k}(-r),E_{j,l}(-s)] \Big)\\
		&= \sum_{r,s \geq 0} \sum_{\substack{1 \leq k \leq i\\ 1 \leq l \leq j}} \\
		&\Bigg( \Big( \delta_{l,i}E_{k,j}(r+s)-\delta_{k,j}E_{l,i}(r+s) \Big)E_{i,k}(-r)E_{j,l}(-s)\\
		&\qquad + E_{k,i}(r) \Big( \delta_{l,k}E_{i,j}(s-r) - \delta_{i,j}E_{l,k}(s-r) -r\delta_{r,s}(\delta_{i,j}\delta_{k,l}c-\delta_{i,k}\delta_{j,l}) \Big) E_{j,l}(-s)\\
		&\qquad\qquad + E_{l,j}(s) \Big( \delta_{j,i}E_{k,l}(r-s) - \delta_{k,l}E_{j,i}(r-s) +r\delta_{r,s}(\delta_{i,j}\delta_{k,l}c-\delta_{i,k}\delta_{j,l}) \Big) E_{i,k}(-r) \\
		&\qquad\qquad\qquad + E_{l,j}(s)E_{k,i}(r)\Big( \delta_{j,k}E_{i,l}(-r-s)-\delta_{i,l}E_{j,k}(-r-s) \Big) \Bigg)\\
		&= \sum_{r,s \geq 0}\sum_{k=1}^i \Big( E_{k,j}(r+s)E_{i,k}(-r)E_{j,i}(-s) + E_{k,i}(r)E_{i,j}(s-r)E_{j,k}(-s) \\
		&\qquad\qquad\qquad - E_{k,j}(s)E_{j,i}(r-s)E_{i,k}(-r) - E_{i,j}(s)E_{k,i}(r)E_{j,k}(-r-s) \Big)\\
		&\quad + \sum_{r>0} r\Big( E_{i,i}(r)E_{j,j}(-r) - E_{j,j}(r)E_{i,i}(-r) \Big).
	\end{split}
\end{equation*}
The sum of the terms containing
\[
	E_{k,i}(a_1), E_{i,j}(a_2), E_{j,k}(a_3) \ (a_1,a_2,a_3 \in \bbZ)
\]
is
\begin{equation*}
	\begin{split}
		&\sum_{r,s \geq 0}\sum_{k=1}^i \Big( E_{k,i}(r)E_{i,j}(s-r)E_{j,k}(-s) - E_{i,j}(s)E_{k,i}(r)E_{j,k}(-r-s) \Big)\\
		&= \sum_{r,s \geq 0}\sum_{k=1}^i \Big( E_{k,i}(r)E_{i,j}(s-r)E_{j,k}(-s) - E_{k,i}(r)E_{i,j}(s)E_{j,k}(-r-s) \\
		&\qquad\qquad\qquad\qquad\qquad\qquad\qquad\qquad\qquad - [E_{i,j}(s),E_{k,i}(r)]E_{j,k}(-r-s) \Big)\\
		&= \sum_{\substack{r,s \geq 0\\r>s}}\sum_{k=1}^i E_{k,i}(r)E_{i,j}(s-r)E_{j,k}(-s) + \sum_{r,s \geq 0}\sum_{k=1}^i E_{k,j}(r+s)E_{j,k}(-r-s).
	\end{split}
\end{equation*}
%
Similarly the sum of the terms containing
\[
	E_{k,j}(a_1), E_{j,i}(a_2), E_{i,k}(a_3) \ (a_1,a_2,a_3 \in \bbZ)
\]
is
\begin{equation*}
	-\sum_{\substack{r,s \geq 0\\r>s}}\sum_{k=1}^i E_{k,j}(s)E_{j,i}(r-s)E_{i,k}(-r) - \sum_{r,s \geq 0}\sum_{k=1}^i E_{k,j}(r+s)E_{j,k}(-r-s).
\end{equation*}
Hence the first identity holds.

The second identity $[A_i,B_j]=0$ is clear due to the condition $i < j$.

A direct computation shows
\begin{equation*}
	\begin{split}
		&[B_i,A_j] \\
		&= \sum_{r,s \geq 0} \Bigg( \sum_{k=i+1}^N E_{k,j}(r+s+1)E_{i,k}(-r-1)E_{j,i}(-s) - \sum_{k=1}^j E_{k,i}(r+s+1)E_{i,j}(-r-1)E_{j,k}(-s) \\
		&\qquad\qquad\quad + \sum_{k=i+1}^j \Big( E_{k,i}(r+1)E_{i,j}(s-r-1)E_{j,k}(-s) - E_{k,j}(s)E_{j,i}(r-s+1)E_{i,k}(-r-1) \Big) \\
		&\qquad\qquad\qquad +\sum_{k=1}^j E_{k,j}(s)E_{j,i}(r+1)E_{i,k}(-r-s-1) - \sum_{k=i+1}^N E_{i,j}(s)E_{k,i}(r+1)E_{j,k}(-r-s-1) \Bigg).
	\end{split}
\end{equation*}
The sums of the terms containing
\[
	E_{k,i}(a_1), E_{i,j}(a_2), E_{j,k}(a_3) \ (a_1,a_2,a_3 \in \bbZ), \quad E_{k,j}(a_1), E_{j,i}(a_2), E_{i,k}(a_3) \ (a_1,a_2,a_3 \in \bbZ)
\]
are
\begin{equation*}
	\begin{split}
		&\sum_{r,s \geq 0} \Bigg( -\sum_{k=1}^i E_{k,i}(r+s+1)E_{i,j}(-r-1)E_{j,k}(-s) - \sum_{k=j+1}^N E_{k,i}(r+1)E_{i,j}(s)E_{j,k}(-r-s-1)\\
		&\qquad\qquad + \sum_{k=i+1}^N E_{k,j}(r+s+1)E_{j,k}(-r-s-1) \Bigg) - \sum_{r > 0} r E_{i,i}(r)E_{j,j}(-r),
	\end{split}
\end{equation*}
\begin{equation*}
	\begin{split}
		&\sum_{r,s \geq 0} \Bigg( \sum_{k=j+1}^N E_{k,j}(r+s+1)E_{j,i}(-s)E_{i,k}(-r-1) +\sum_{k=1}^i E_{k,j}(s)E_{j,i}(r+1)E_{i,k}(-r-s-1)\\
		&\qquad\qquad - \sum_{k=i+1}^N E_{k,j}(r+s+1)E_{j,k}(-r-s-1) \Bigg) + \sum_{r > 0} r E_{j,j}(r)E_{i,i}(-r),
	\end{split}
\end{equation*}
respectively.
The former follows from
\begin{equation*}
	\begin{split}
		&\sum_{r,s \geq 0} \Bigg( - \sum_{k=1}^j E_{k,i}(r+s+1)E_{i,j}(-r-1)E_{j,k}(-s) + \sum_{k=i+1}^j E_{k,i}(r+1)E_{i,j}(s-r-1)E_{j,k}(-s)\\
		&\qquad\qquad - \sum_{k=i+1}^N E_{i,j}(s)E_{k,i}(r+1)E_{j,k}(-r-s-1) \Bigg)\\
		&= \sum_{r,s \geq 0} \Bigg( - \sum_{k=1}^j E_{k,i}(r+s+1)E_{i,j}(-r-1)E_{j,k}(-s) + \sum_{k=i+1}^j E_{k,i}(r+1)E_{i,j}(s-r-1)E_{j,k}(-s) \\
		&\qquad\qquad - \sum_{k=i+1}^N \Big( E_{k,i}(r+1)E_{i,j}(s)E_{j,k}(-r-s-1) + [E_{i,j}(s),E_{k,i}(r+1)]E_{j,k}(-r-s-1) \Big)\Bigg)\\
		&= \sum_{r,s \geq 0} \Bigg( -\sum_{k=1}^i E_{k,i}(r+s+1)E_{i,j}(-r-1)E_{j,k}(-s) - \sum_{k=j+1}^N E_{k,i}(r+1)E_{i,j}(s)E_{j,k}(-r-s-1)\\
		&\qquad\qquad - \sum_{k=i+1}^N \Big( \delta_{k,j}E_{i,i}(r+s+1) - E_{k,j}(r+s+1) \Big) E_{j,k}(-r-s-1) \Bigg)\\
		&= \sum_{r,s \geq 0} \Bigg( -\sum_{k=1}^i E_{k,i}(r+s+1)E_{i,j}(-r-1)E_{j,k}(-s) - \sum_{k=j+1}^N E_{k,i}(r+1)E_{i,j}(s)E_{j,k}(-r-s-1)\\
		&\qquad\qquad + \sum_{k=i+1}^N E_{k,j}(r+s+1)E_{j,k}(-r-s-1) \Bigg) - \sum_{r > 0} r E_{i,i}(r)E_{j,j}(-r),
	\end{split}
\end{equation*}
and the latter is similarly obtained.
Hence the third identity holds.

A similar computation shows the last identity.

%
%
%
\end{proof}
We show (\ref{eq:mainthm+-}).
Assume $i,j \neq 0$.
The proofs for the remaining cases are similar.
We have
\begin{equation*}
	\begin{split}
		&[\ev({x}_{i,1}^+), {x}_{j}^{-}] = (1+i \ve_2)\delta_{ij}h_i \\
		&\quad + \hbar \sum_{s \geq 0} \Big( \sum_{k=1}^i[E_{k,i+1}(s)E_{i,k}(-s), E_{j+1,j}] + \sum_{k=i+1}^N[E_{k,i+1}(s+1)E_{i,k}(-s-1), E_{j+1,j}] \Big).
	\end{split}
\end{equation*}
Since we have
\begin{equation*}
	\begin{split}
		&[E_{k,i+1}(s)E_{i,k}(-s), E_{j+1,j}] \\
		&= \Big(\delta_{i,j}E_{k,j}(s)-\delta_{j,k}E_{j+1,i+1}(s)\Big)E_{i,k}(-s) + E_{k,i+1}(s)\Big(\delta_{j+1,k}E_{i,j}(-s)-\delta_{i,j}E_{j+1,k}(-s)\Big),
	\end{split}
\end{equation*}
we see
\begin{equation*}
	\begin{split}
		&\sum_{k=1}^i[E_{k,i+1}(s)E_{i,k}(-s), E_{j+1,j}] \\
		&=\sum_{k=1}^i \delta_{i,j}E_{k,j}(s)E_{i,k}(-s) -\delta(j \leq i)E_{j+1,i+1}(s) E_{i,j}(-s) \\
		&\quad + \delta(j+1 \leq i)E_{j+1,i+1}(s) E_{i,j}(-s)- \sum_{k=1}^i \delta_{i,j}E_{k,i+1}(s)E_{j+1,k}(-s)\\
		&=\delta_{i,j} \Big( \sum_{k=1}^i \Big( E_{k,i}(s)E_{i,k}(-s) - E_{k,i+1}(s)E_{i+1,k}(-s) \Big) -E_{i+1,i+1}(s) E_{i,i}(-s) \Big) 
	\end{split}
\end{equation*}
and
\begin{equation*}
	\begin{split}
		&\sum_{k=i+1}^N[E_{k,i+1}(s+1)E_{i,k}(-s-1), E_{j+1,j}] \\
		&=\sum_{k=i+1}^N \delta_{i,j}E_{k,j}(s+1)E_{i,k}(-s-1) -\delta(j \geq i+1)E_{j+1,i+1}(s+1) E_{i,j}(-s-1) \\
		&\quad + \delta(j \geq i)E_{j+1,i+1}(s+1) E_{i,j}(-s-1)- \sum_{k=i+1}^N \delta_{i,j}E_{k,i+1}(s+1)E_{j+1,k}(-s-1)\\
		&=\delta_{i,j} \Big( \sum_{k=i+1}^N \Big( E_{k,i}(s+1)E_{i,k}(-s-1) - E_{k,i+1}(s+1)E_{i+1,k}(-s-1) \Big) + E_{i+1,i+1}(s+1) E_{i,i}(-s-1) \Big).
	\end{split}
\end{equation*}
Therefore
\begin{equation*}
	\begin{split}
		&[\ev({x}_{i,1}^+), {x}_{j}^{-}] \\
		&= \delta_{i,j} \Bigg( (1+i \ve_2)h_i -\hbar E_{i+1,i+1} E_{i,i} + \hbar \sum_{s \geq 0} \Bigg( \sum_{k=1}^i \Big( E_{k,i}(s)E_{i,k}(-s) - E_{k,i+1}(s)E_{i+1,k}(-s) \Big) \\
		&\qquad\qquad\qquad\qquad\qquad\qquad + \sum_{k=i+1}^N \Big( E_{k,i}(s+1)E_{i,k}(-s-1) - E_{k,i+1}(s+1)E_{i+1,k}(-s-1) \Big) \Bigg) \Bigg) \\
		&= \delta_{i,j} \ev(h_{i,1}).
	\end{split}
\end{equation*}

We show (\ref{eq:mainthmHX}) for $+$.
Assume $i,j \neq 0$.
The right-hand side of (\ref{eq:mainthmHX}) is
\begin{equation*}
	\begin{split}
		&a_{ij} \Bigg( \Big( 1+i \ve_2 + (\delta_{i+1,j} - \delta_{i-1,j}) \dfrac{\hbar}{2} \Big) x_j^+ \\
		&\qquad\qquad + \hbar\sum_{s \geq 0} \Big( \sum_{k=1}^j E_{k,j+1}(s) E_{j,k}(-s) + \sum_{k=j+1}^N E_{k,j+1}(s+1) E_{j,k}(-s-1) \Big) \Bigg).
	\end{split}
\end{equation*}
A direct computation shows
\begin{equation*}
	\begin{split}
		&[\ev(\tilde{h}_{i,1}), x_{j}^{+}] = (1+ i \varepsilon_2) a_{ij} x_{j}^+ \\
		&\ - \dfrac{\hbar}{2} \Big( \delta_{i,j}(\{E_{i,i+1}, E_{i,i}\} - \{E_{i,i+1}, E_{i+1,i+1}\}) - \delta_{i,j+1}\{E_{i-1,i}, E_{i,i}\} + \delta_{i+1,j}\{E_{i+1,i+2}, E_{i+1,i+1}\} \Big)\\
		&\ \quad + \hbar \Big( \delta_{i,j}(E_{i,i}E_{i,i+1} - E_{i,i+1}E_{i+1,i+1}) - \delta_{i,j+1} E_{i,i}E_{i-1,i} + \delta_{i+1,j} E_{i+1,i+2}E_{i+1,i+1} \Big)\\
		&\ \qquad + \hbar \sum_{s \geq 0} \Bigg( 2\delta_{i,j} \Big( \sum_{k=1}^i E_{k,i+1}(s)E_{i,k}(-s) + \sum_{k=i+1}^N E_{k,i+1}(s+1)E_{i,k}(-s-1)  \Big)\\
		&\quad\qquad\qquad\qquad - \delta_{i,j+1} \Big( \sum_{k=1}^{i-1} E_{k,i}(s)E_{i-1,k}(-s) + \sum_{k=i}^N E_{k,i}(s+1)E_{i-1,k}(-s-1) \Big)\\
		&\qquad\qquad\qquad\qquad - \delta_{i+1,j} \Big( \sum_{k=1}^{i+1} E_{k,i+2}(s)E_{i+1,k}(-s) + \sum_{k=i+2}^N E_{k,i+2}(s+1)E_{i+1,k}(-s-1) \Big) \Bigg). 
	\end{split}
\end{equation*}
Hence the desired identity follows from
\begin{equation*}
	\begin{split}
		&\dfrac{\hbar}{2} \Big(\!-\{E_{i,i+1}, E_{i,i}\} + \{E_{i,i+1}, E_{i+1,i+1}\} + 2E_{i,i}E_{i,i+1} - 2E_{i,i+1}E_{i+1,i+1}\Big) \\
		&=\dfrac{\hbar}{2}\Big([E_{i,i},E_{i,i+1}] + [E_{i+1,i+1},E_{i,i+1}]\Big) = 0
	\end{split}
\end{equation*}
for the case $i=j$,
\begin{equation*}
	\begin{split}
		&\dfrac{\hbar}{2} \Big(\{E_{i-1,i}, E_{i,i}\} - 2E_{i,i}E_{i-1,i}\Big) = \dfrac{\hbar}{2} [E_{i-1,i}, E_{i,i}] = \dfrac{\hbar}{2} x_{i-1}^+
	\end{split}
\end{equation*}
for the case $i=j+1$,
\begin{equation*}
	\begin{split}
		&\dfrac{\hbar}{2} \Big(\!-\{E_{i+1,i+2}, E_{i+1,i+1}\} + 2E_{i+1,i+2}E_{i+1,i+1}\Big) = \dfrac{\hbar}{2} [E_{i+1,i+2}, E_{i+1,i+1}] = -\dfrac{\hbar}{2} x_{i+1}^+
	\end{split}
\end{equation*}
for the case $i=j-1$.
We thus have proved the case $i,j \neq 0$.
The proofs for the remaining cases are similar, but we give computations for the cases $i=1$, $j=0$ and $i=0$, $j=1$ to clarify a role of the condition $\hbar c = N\ve_2$.
First assume $i=1$, $j=0$.
Then the right-hand side of (\ref{eq:mainthmHX}) is
\begin{equation*}
	\begin{split}
		&- \Bigg( \Big( 1+ \ve_2 + N \ve_2 - \dfrac{\hbar}{2} \Big) x_0^+ + \hbar\sum_{s \geq 0} \sum_{k=1}^N E_{k,1}(s+1) E_{N,k}(-s) \Bigg).
	\end{split}
\end{equation*}
The left-hand side is
\begin{equation*}
	\begin{split}
		&[\ev(\tilde{h}_{1,1}), x_{0}^{+}] \\
		&= -(1+ \varepsilon_2) x_{0}^+ + \dfrac{\hbar}{2} \{E_{N,1}(1), E_{1,1}\} -  \hbar c x_0^+ - \hbar E_{1,1}E_{N,1}(1) - \hbar \sum_{s \geq 0} \sum_{k=1}^N E_{k,1}(s+1)E_{N,k}(-s)\\
		&= - \Bigg( (1+ \ve_2)x_0^+  + \hbar c x_0^+ - \dfrac{\hbar}{2} \Big( \{E_{N,1}(1), E_{1,1}\}-2E_{1,1}E_{N,1}(1) \Big) + \hbar\sum_{s \geq 0} \sum_{k=1}^N E_{k,1}(s+1) E_{N,k}(-s) \Bigg). 
	\end{split}
\end{equation*}
Hence the assertion holds.
Next assume $i = 0$, $j = 1$.
Then the right-hand side of (\ref{eq:mainthmHX}) is
\begin{equation*}
	\begin{split}
		&- \Bigg( \Big( 1+ \dfrac{\hbar}{2} \Big) x_1^+ + \hbar\sum_{s \geq 0} \Big( E_{1,2}(s) E_{1,1}(-s) + \sum_{k=2}^N E_{k,2}(s+1) E_{1,k}(-s-1) \Big) \Bigg).
	\end{split}
\end{equation*}
The left-hand side is
\begin{equation*}
	\begin{split}
		&[\ev(\tilde{h}_{0,1}), x_{1}^{+}] \\
		&= -(1+ N \varepsilon_2) x_{1}^+ - \dfrac{\hbar}{2} \Big( \{E_{1,2}, E_{1,1}\}-2E_{1,2}c \Big) - \hbar \sum_{s \geq 0} \sum_{k=1}^N E_{k,2}(s+1) E_{1,k}(-s-1)\\
		&= -\Bigg( (1+ N \varepsilon_2) x_1^+  - \hbar c x_1^+ + \dfrac{\hbar}{2} \Big( \{E_{1,2}, E_{1,1}\} - 2 E_{1,2}E_{1,1} \Big)\\
		&\qquad\qquad\qquad\qquad + \hbar \sum_{s \geq 0} \Big( E_{1,2}(s)E_{1,1}(-s) + \sum_{k=2}^{N} E_{k,2}(s+1)E_{1,k}(-s-1) \Big)\Bigg).		
	\end{split}
\end{equation*}
Hence the assertion holds.

We show (\ref{eq:mainthmXX2}) for $+$.
First assume $i,j \neq 0$.
A direct computation shows
\begin{equation*}
	\begin{split}
		&[\ev(x_{i, 1}^{+}), x_{j}^{+}] = (1+ i \varepsilon_2) [x_{i}^{+},x_{j}^{+}] + \hbar \delta_{i,j} (x_i^+)^2\\
		&\qquad + \hbar \sum_{s \geq 0} \Bigg( \delta_{i+1,j} \Big( \sum_{k=1}^i   E_{k,i+2}(s) E_{i,k}(-s) + \sum_{k=i+1}^N E_{k,i+2}(s+1) E_{i,k}(-s-1) \Big) \\
		&\qquad\qquad\qquad - \delta_{i,j+1} \Big( \sum_{k=1}^i E_{k,i+1}(s)E_{i-1,k}(-s) + \sum_{k=i+1}^N E_{k,i+1}(s+1)E_{i-1,k}(-s-1) \Big)\Bigg).
	\end{split}
\end{equation*}
By swapping $i$ and $j$, and summing up, we obtain
\begin{equation*}
	\begin{split}
		&[\ev(x_{i, 1}^{+}), x_{j}^{+}] - [x_{i}^{+}, \ev(x_{j, 1}^{+})]= (i-j)\varepsilon_2 [x_{i}^{+},x_{j}^{+}] + 2 \hbar \delta_{i,j} (x_{i}^+)^2\\
		&\qquad + \hbar \Big( \delta_{i+1,j} (-x_{i+1}^+ x_{i}^+) + \delta_{i,j+1} (-x_i^+ x_{i-1}^+) \Big) \\
		&= \begin{cases}
			2 \hbar (x_{i}^+)^2 & j=i,\\
			-\ve_2 x_i^+ x_{i+1}^+ - \ve_1 x_{i+1}^+ x_{i}^+ & j=i+1,\\
			-\ve_1 x_i^+ x_{i-1}^+ - \ve_2 x_{i-1}^+ x_{i}^+ & j=i-1,\\
			0 & \text{otherwise}.
		\end{cases}
	\end{split}
\end{equation*}
These are nothing but the desired relations.

Next consider the other cases.
For the case $i=j=0$, we have
\begin{equation*}
	\begin{split}
		&[\ev(x_{0, 1}^{+}), x_{0}^{+}] = \hbar \sum_{s \geq 0} \sum_{k=1}^N [E_{k,1}(s+1) E_{N,k}(-s), E_{N,1}(1)]\\
		&= \hbar \sum_{s \geq 0} \Big(\! -E_{N,1}(s+2) E_{N,1}(-s) + E_{N,1}(s+1) E_{N,1}(-s+1) \Big) = \hbar (x_{0}^+)^2.
	\end{split}
\end{equation*}
For the case $i=0$, $j \neq 0$, we have
\begin{equation*}
	\begin{split}
		&[\ev(x_{0, 1}^{+}), x_{j}^{+}] - [x_{0}^{+}, \ev(x_{j, 1}^{+})]= (N-j)\varepsilon_2 [x_{0}^{+},x_{j}^{+}] - \hbar c \delta_{1,j} [x_{0}^{+},x_{1}^{+}]\\
		&\qquad + \hbar \Big( \delta_{1,j} (-x_{1}^+ x_{0}^+) + \delta_{N,j+1} (-x_0^+ x_{N-1}^+) \Big) \\
		&= \begin{cases}
			-\ve_2 x_0^+ x_{1}^+ - \ve_1 x_{1}^+ x_{0}^+ & j=1,\\
			-\ve_1 x_0^+ x_{N-1}^+ - \ve_2 x_{N-1}^+ x_{0}^+ & j=N-1,\\
			0 & \text{otherwise}.
		\end{cases}
	\end{split}
\end{equation*}
We need the condition $\hbar c = N \ve_2$ for the case $i=0$, $j=1$. 
\end{proof}

\subsection{Cyclic automorphism}\label{subsection:cyclic}

We consider algebra automorphisms corresponding to the rotation of the Dynkin diagram.
It is easy to see that the assignment
\[
	x_i^{\pm} \mapsto x_{i-1}^{\pm}, \quad h_i \mapsto h_{i-1}, \quad c \mapsto c, \quad \bfid(s) \mapsto \bfid(s) + \delta_{s,0}c
\]
gives an algebra automorphism $\rho_U$ of $U(\hat{\gl}_N)$.
Guay introduced an analogous automorphism for the affine Yangian.
\begin{prop}[\cite{MR2199856}, Lemma~3.5]
The assignment
\[
	x_{i,r}^{\pm} \mapsto \sum_{s=0}^r \dbinom{r}{s} \ve_2^{r-s} x_{i-1,s}^{\pm}, \quad h_{i,r} \mapsto \sum_{s=0}^r \dbinom{r}{s} \ve_2^{r-s} h_{i-1,s}
\]
gives an algebra automorphism $\rho$ of $\affY$.
\end{prop}

\begin{lem}
We have $\rho_U(E_{ij}(s)) = E_{i-1,j-1}(s+\delta_{i,1}-\delta_{j,1}) + \delta_{s,0}\delta_{i,1}\delta_{j,1}c$.
\end{lem}
\begin{proof}
See \cite[Lemma~4.1]{MR3923494}.
\end{proof}

\begin{prop}
We have $\rho_{U} \circ \ev = \ev \circ \rho$.
\end{prop}
\begin{proof}
The identity obviously holds for $x_{i,0}^{\pm}$, $h_{i,0}$.
Hence it is enough to show that the identity also holds for $x_{i,1}^+$ since $\affY$ is generated by  $x_{i,0}^{\pm}$, $h_{i,0}$, $x_{i,1}^+$ ($i \in \bbZ/N\bbZ$). 
We show $\rho_U(\ev(x_{i,1}^+)) = \ev(\rho(x_{i,1}^+))$ for $i \neq 0$.
The proof for the case $i = 0$ is similar.
The left-hand side is
\begin{equation*}
	\begin{split}
		&\rho_U(\ev(x_{i,1}^+)) = (1+ i \varepsilon_2) x_{i-1}^+ \\
		&\quad + \hbar \sum_{s \geq 0} \Big( \sum_{k=1}^i \rho_U(E_{k,i+1}(s) E_{i,k}(-s)) + \sum_{k=i+1}^N \rho_U(E_{k,i+1}(s+1) E_{i,k}(-s-1)) \Big).
	\end{split}
\end{equation*}
We have 
\begin{equation*}
	\begin{split}
		&\sum_{k=1}^i \rho_U(E_{k,i+1}(s) E_{i,k}(-s)) =\sum_{k=1}^i  E_{k-1,i}(s+\delta_{k,1}) \Big( E_{i-1,k-1}(-s+\delta_{i,1}-\delta_{k,1}) + \delta_{s,0}\delta_{i,1}\delta_{k,1} c \Big) \\
		&= E_{N,i}(s+1) E_{i-1,N}(-s-1+\delta_{i,1}) + \delta_{s,0}\delta_{i,1} E_{N,1}(1)c + \sum_{k=2}^i E_{k-1,i}(s) E_{i-1,k-1}(-s+\delta_{i,1})
	\end{split}
\end{equation*}
and
\begin{equation*}
	\begin{split}
		&\sum_{k=i+1}^{N} \rho_U(E_{k,i+1}(s+1) E_{i,k}(-s-1)) =\sum_{k=i+1}^{N} E_{k-1,i}(s+1) E_{i-1,k-1}(-s-1+\delta_{i,1}).
	\end{split}
\end{equation*}
This verifies that $\rho_U(\ev(x_{i,1}^+))$ is equal to $\ev(\rho(x_{i,1}^+))=\ev(x_{i-1,1}^+)+\ve_2 x_{i-1}^+$ under the assumption $\hbar c = N \ve_2$.
\end{proof}

\subsection{Another version}

Guay's evaluation map $\ev$ has $U(\hat{\gl}_N)_{{\rm comp},-}$ as its target space and hence it matches with lowest weight modules of $\hat{\gl}_N^{(-)}$.
To deal with highest weight modules, we introduce an opposite evaluation map by using the anti-isomorphisms $\mu$ and $\mu_U$.
Moreover we make it one-parameter family of algebra homomorphisms by using $\tau_{\alpha}$.
\begin{dfn}
For each $\alpha \in \bbC$, define $\ev_\alpha^+ = \mu_U^{-1} \circ \ev \circ \tau_{\alpha-1} \circ \mu$.
Here $\ev$ and $\tau_{\alpha-1}$ are taken as those for $Y_{-\ve_2,-\ve_1}(\affsl)$.
\end{dfn}

\begin{thm}\label{thm:main}
Assume $\hbar c = -N \ve_1$ and let $\alpha$ be a complex number.
Then there exists an algebra homomorphism $\ev^+_{\alpha} \colon \affY \to U(\hat{\gl}_N)_{{\rm comp},+}$ uniquely determined by 
\begin{gather*}
	\ev^+_{\alpha}(x_{i,0}^{+}) = x_{i}^{+}, \quad \ev^+_{\alpha}(x_{i,0}^{-}) = x_{i}^{-},\quad \ev^+_{\alpha}(h_{i,0}) = h_{i},
\end{gather*}
\begin{gather*}
	\ev^+_{\alpha}(x_{i,1}^{+}) = \begin{cases}
		(\alpha - N \varepsilon_1) x_{0}^{+} + \hbar \displaystyle\sum_{s \geq 0} \sum_{k=1}^N E_{N,k}(-s) E_{k,1}(s+1) \text{ if $i = 0$},\\
		(\alpha - i \varepsilon_1) x_{i}^{+} + \hbar \displaystyle\sum_{s \geq 0} \Big( \sum_{k=1}^i E_{i,k}(-s) E_{k,i+1}(s) + \sum_{k=i+1}^N E_{i,k}(-s-1) E_{k,i+1}(s+1) \Big)\\
		\qquad\qquad\qquad\qquad\qquad\qquad\qquad\qquad\qquad\qquad\qquad\qquad\qquad\qquad \text{ if $i \neq 0$},
	\end{cases}
\end{gather*}
\begin{gather*}
	\ev^+_{\alpha}(x_{i,1}^-) = \begin{cases}
		(\alpha - N \varepsilon_1) x_{0}^{-} + \hbar \displaystyle\sum_{s \geq 0} \sum_{k=1}^N E_{1,k}(-s-1) E_{k,N}(s) \text{ if $i = 0$},\\
		(\alpha - i \varepsilon_1) x_{i}^{-} + \hbar \displaystyle\sum_{s \geq 0} \Big( \sum_{k=1}^i E_{i+1,k}(-s) E_{k,i}(s) + \sum_{k=i+1}^N E_{i+1,k}(-s-1) E_{k,i}(s+1) \Big) \\
		\qquad\qquad\qquad\qquad\qquad\qquad\qquad\qquad\qquad\qquad\qquad\qquad\qquad\qquad\text{ if $i \neq 0$},
	\end{cases}
\end{gather*}
\begin{gather*}
	\ev^+_{\alpha}(h_{i,1}) = \begin{cases}
		(\alpha - N \varepsilon_1) h_{0} - \hbar E_{N,N} (E_{1,1}-c) \\
		\quad + \hbar \displaystyle\sum_{s \geq 0} \sum_{k=1}^{N} \Big( E_{N,k}(-s) E_{k,N}(s) - E_{1,k}(-s-1) E_{k,1}(s+1) \Big) \text{ if $i = 0$},\\
\\
		(\alpha - i \varepsilon_1) h_{i} - \hbar E_{i,i}E_{i+1,i+1} \\
		\quad + \hbar \displaystyle\sum_{s \geq 0} \Big( \sum_{k=1}^{i} E_{i,k}(-s) E_{k,i}(s) + \displaystyle\sum_{k=i+1}^{N} E_{i,k}(-s-1) E_{k,i}(s+1) \\
		\qquad\qquad\quad - \displaystyle\sum_{k=1}^{i} E_{i+1,k}(-s) E_{k,i+1}(s) - \displaystyle\sum_{k=i+1}^{N} E_{i+1,k}(-s-1) E_{k,i+1}(s+1) \Big) \\
		\qquad\qquad\qquad\qquad\qquad\qquad\qquad\qquad\qquad\qquad\qquad\qquad\qquad\qquad \text{ if $i \neq 0$}.
	\end{cases}
\end{gather*}
\end{thm}

\section{Evaluation modules}

In this section, the symbol $\hat{\gl}_N$ denotes $\hat{\gl}_N^{(+)}$ introduced in Section~\ref{subsection:affine_Lie_algebra}.

Let $L(\Lambda)$ be the integrable irreducible highest weight module of $\hat{\gl}_N$ with highest weight $\Lambda$.
The dominant integral weight $\Lambda$ is determined by the data $\lambda_1, \ldots, \lambda_N \in \bbC$ and $K \in \bbZ_{\geq 1}$ satisfying $\lambda_i-\lambda_{i+1} \in \bbZ_{\geq 0}$ ($1 \leq i \leq N-1$) and $\lambda_N-\lambda_1+K \in \bbZ_{\geq 0}$.
The correspondence is given by
\[
	\langle E_{i,i}, \Lambda \rangle = \lambda_i \ (1 \leq i \leq N) \quad \text{ and } \quad \langle c, \Lambda \rangle = K.
\]
We denote a fixed highest weight vector of $L(\Lambda)$ by $v_{\Lambda}$.
 
Assume $K \hbar = -N \ve_1$.
We define a $\affY$-module $L(\Lambda,\alpha)$ by the pull-back of $L(\Lambda)$ via $\ev^+_{\alpha}$.
By abuse of notation, we regard $v_{\Lambda}$ as a vector of $L(\Lambda,\alpha)$.
The vector $v_{\Lambda}$ satisfies the condition of the highest weight vector for affine Yangian.
More precisely we have the following.
\begin{thm}\label{thm:highest_weight}
We have
\begin{gather*}
	x_{i,r}^+ v_{\Lambda} = 0,\\
	h_{i,r} v_{\Lambda} = a_i^r \langle h_{i}, \Lambda \rangle v_{\Lambda}
\end{gather*}
for all $i \in \bbZ/N\bbZ$ and $r \geq 0$, where
\begin{equation*}
	\begin{split}
		a_i &= \begin{cases}
			\alpha - N\ve_1 + \hbar \lambda_N & \text{if $i=0$},\\
			\alpha - i\ve_1 + \hbar \lambda_i & \text{otherwise},
		\end{cases}\\
		&= \begin{cases}
			\alpha + (\lambda_N+K)\hbar & \text{if $i=0$},\\
			\alpha + \left(\lambda_i+\dfrac{i}{N}K \right)\hbar & \text{otherwise}.
	\end{cases}
	\end{split}
\end{equation*}
\end{thm}
The goal of the remaining part is to prove this theorem.

\begin{rem}\label{rem:thm}
In \cite{MR3923494}, the author proves that the image of Guay's evaluation map $\ev$ in $U(\hat{\gl}_N)_{{\rm comp},-}$ contains $\hat{\gl}_N^{(-)}$ under the assumption $\ve_2 \neq 0$.
This result implies that the image of $\ev^+_{\alpha}$ in $U(\hat{\gl}_N)_{{\rm comp},+}$ contains $\hat{\gl}_N^{(+)}$ under the assumption $\ve_1 \neq 0$.
Hence the $\affY$-module $L(\Lambda,\alpha)$ is irreducible when $\ve_1 \neq 0$.
We do not use this fact in the proof given below.
\end{rem}

We provide a general lemma.
\begin{lem}\label{lem:general_lemma}
Let $v$ be a nonzero element of a $\affY$-module satisfying
\begin{equation}
	x_{i,0}^+ v = 0, \quad h_{i,0} v = P_{i,0} v \label{eq:assumption1}
\end{equation}
for some $P_{i,0} \in \bbC$.
Further assume that $v$ satisfies
\begin{equation}
	h_{i,1} v = P_{i,1} v, \label{eq:assumption2}
\end{equation}
\begin{equation}
	x_{i,1}^- v = Q_{i} x_{i,0}^- v \label{eq:assumption3}
\end{equation}
for some $P_{i,1}, Q_{i} \in \bbC$.
Then $v$ satisfies
\[
	x_{i,r}^+ v = 0,\quad h_{i,r} v = P_{i,r} v
\]
for all $r \geq 0$, where $P_{i,r}$ $(r \geq 2)$ is given by
\begin{equation}
	P_{i,r} = Q_{i}^{r-1} P_{i,1}. \label{eq:eigen}
\end{equation}
\end{lem}

\begin{proof}
We prove $x_{i,r}^+ v = 0$ by induction on $r$.
The assertion for $r=0$ holds by the assumption~(\ref{eq:assumption1}).
Assume it for $r$.
Then we have
\begin{equation*}
	\begin{split}
		x_{i,r+1}^+ v = \dfrac{1}{2} [\tilde{h}_{i,1}, x_{i,r}^+] v = -\dfrac{1}{2} x_{i,r}^+ \tilde{h}_{i,1} v = -\dfrac{1}{2} x_{i,r}^+ \left( h_{i,1}-\dfrac{\hbar}{2}h_{i,0}^2 \right) v
	\end{split}
\end{equation*}	
by the induction assumption, and it is equal to $0$ by the assumptions~(\ref{eq:assumption1}), (\ref{eq:assumption2}) and again by the induction assumption. 

We prove that $v$ is an eigenvector of $h_{i,r}$ with the eigenvalue (\ref{eq:eigen}) by induction on $r$.
The assertion for $r=1$ holds by the assumption (\ref{eq:assumption2}).
Assume it for $r$.
We have
\begin{equation*}
	\begin{split}
		h_{i,r+1} v &= [x_{i,r+1}^+,x_{i,0}^-] v = x_{i,r+1}^+ x_{i,0}^- v = \dfrac{1}{2}[\tilde{h}_{i,1},x_{i,r}^+] x_{i,0}^- v\\
		&= \dfrac{1}{2} \left( \tilde{h}_{i,1} x_{i,r}^+ x_{i,0}^- v - x_{i,r}^+ \tilde{h}_{i,1} x_{i,0}^- v \right).
	\end{split}
\end{equation*}
Here we use $x_{i,r+1}^+ v = 0$ in the second equality.	Then we have
\begin{equation*}
	\begin{split}
		\tilde{h}_{i,1} x_{i,r}^+ x_{i,0}^- v &= \tilde{h}_{i,1} ([x_{i,r}^+, x_{i,0}^-] + x_{i,0}^- x_{i,r}^+ )v \\
		&= \tilde{h}_{i,1} h_{i,r} v \quad \text{by $x_{i,r}^+ v=0$} \\
		&= \left(P_{i,1}-\dfrac{\hbar}{2}P_{i,0}^2 \right) P_{i,r} v \quad \text{by (\ref{eq:assumption1}), (\ref{eq:assumption2}), the induction assumption,}
	\end{split}
\end{equation*}
and
\begin{equation*}
	\begin{split}
		x_{i,r}^+ \tilde{h}_{i,1} x_{i,0}^- v &= x_{i,r}^+ ([\tilde{h}_{i,1}, x_{i,0}^-] + x_{i,0}^- \tilde{h}_{i,1} )v \\
		&= -2 x_{i,r}^+ x_{i,1}^- v + \left(P_{i,1}- \dfrac{\hbar}{2} P_{i,0}^2\right)x_{i,r}^+ x_{i,0}^- v \quad \text{by (\ref{eq:assumption1}), (\ref{eq:assumption2})}\\
		&= \left( -2Q_{i} + P_{i,1}- \dfrac{\hbar}{2} P_{i,0}^2\right)x_{i,r}^+ x_{i,0}^- v \quad \text{by (\ref{eq:assumption3})}\\
		&= \left( -2Q_{i} + P_{i,1}- \dfrac{\hbar}{2} P_{i,0}^2 \right)h_{i,r} v \quad \text{by $x_{i,r}^+ v=0$}\\
		&= \left( -2Q_{i} + P_{i,1}- \dfrac{\hbar}{2} P_{i,0}^2 \right) P_{i,r} v \quad \text{by the induction assumption.}
	\end{split}
\end{equation*}
Hence we conclude $h_{i,r+1}v = Q_{i} P_{i,r} v$, which completes the proof.
\end{proof}

The vector $v_{\Lambda}$ satisfies the condition (\ref{eq:assumption1}).
In order to apply Lemma~\ref{lem:general_lemma} to our situation, we compute $h_{i,1} v_{\Lambda}$ and $x_{i,1}^- v_{\Lambda}$.
\begin{lem}\label{lem:r=1}
We have $h_{i,1} v_{\Lambda} = a_i \langle h_{i}, \Lambda \rangle v_{\Lambda}$ for all $i \in \bbZ/N\bbZ$.
\end{lem}
\begin{proof}
We give a proof for $i\neq 0$.
The case $i=0$ is similar.
Recall the formula for $\ev^+_{\alpha}(h_{i,1})$ and note that
 all the terms in $\ev^+_{\alpha}(h_{i,1})$ annihilate $v_{\Lambda}$ except for those concerning
\[
	h_i v_{\Lambda} = \langle h_{i}, \Lambda \rangle v_{\Lambda},
\]
\[
	E_{i,i}E_{i+1,i+1} v_{\Lambda} = \lambda_i \lambda_{i+1} v_{\Lambda},
\]
\[
	E_{i,i}^2 v_{\Lambda} = \lambda_i^2 v_{\Lambda}.
\]
Then we see that $v_{\Lambda}$ is an eigenvector of $h_{i,1}$ with the eigenvalue
\begin{equation*}
	(\alpha - i\ve_1)\langle h_{i}, \Lambda \rangle -\hbar \lambda_i \lambda_{i+1} + \hbar \lambda_i^2 = (\alpha - i\ve_1 + \hbar \lambda_i)\langle h_{i}, \Lambda \rangle.
\end{equation*}	
\end{proof}

\begin{lem}\label{lem:wt-2}
We have $x_{i,1}^- v_{\Lambda} = a_i x_i^- v_{\Lambda}$ for all $i \in \bbZ/N\bbZ$.
\end{lem}
\begin{proof}
We give a proof for $i\neq 0$.
The case $i=0$ is similar.
Recall the formula for $\ev^+_{\alpha}(x_{i,1}^-)$.
Then we have
\begin{equation*}
	\begin{split}
		x_{i,1}^- v_{\Lambda} = (\alpha - i\ve_1) x_i^- v_{\Lambda} + \hbar E_{i+1,i}E_{i,i} v_{\Lambda} = (\alpha - i\ve_1 + \hbar \lambda_i) x_i^- v_{\Lambda}.
	\end{split}
\end{equation*}	
\end{proof}

\begin{proof}[Proof of Theorem~\ref{thm:highest_weight}]

By Lemma~\ref{lem:general_lemma}, \ref{lem:r=1}, \ref{lem:wt-2}, we see that $x_{i,r}^+ v_{\Lambda}=0$ and $v_{\Lambda}$ is an eigenvector of $h_{i,r}$ with the eigenvalue
\[
	a_i^{r-1} \times a_i \langle h_{i}, \Lambda \rangle = a_i^r \langle h_{i}, \Lambda \rangle.
\]
\end{proof}

Let us consider the generating series
\[
	1 + \hbar \sum_{r=0}^{\infty} a_i^{r} \langle h_{i}, \Lambda \rangle u^{-r-1} \in \bbC[[u^{-1}]]
\]
of the eigenvalues of $h_{i,r}$ on $v_{\Lambda}$.
It is equal to
\[
	\dfrac{u-a_i+\langle h_{i}, \Lambda \rangle \hbar}{u-a_i} = \dfrac{\pi_i(u+\hbar)}{\pi_i(u)}
\]
where we put $\pi_i(u) = (u-a_i)(u-a_i+\hbar)\cdots(u-a_i+(\langle h_{i}, \Lambda \rangle-1)\hbar)$.
The collection $(\pi_i(u))$ is an analog of the Drinfeld polynomials of finite-dimensional irreducible modules for the Yangian associated with a simple Lie algebra.

In \cite{MR3898327, MR3869424}, the author constructed the level-one Fock representation $F$ of $\affY$.
The polynomials ($\pi_i(u)$) corresponding to its highest weight is
\[
	\pi_i(u)=\begin{cases}
	u & \text{if $i=0$},\\
	1 & \text{otherwise}
	\end{cases}
\]  
by \cite[Theorem~5.7]{MR3898327} (the highest weight vector is denoted by $b_{\varnothing}$).
We see that the level-one Fock representation at $\hbar = -N\ve_1$ with $\ve_1 \neq 0$ is an example of evaluation modules.
Indeed, take $\Lambda$ as the level-one dominant integral weight corresponding to $\lambda_1 = \cdots = \lambda_N=0$ and put $\alpha=-\hbar$.
Then the highest weights of $F$ and $L(\Lambda, -\hbar)$ are the same.
As we mentioned in Remark~\ref{rem:thm}, $L(\Lambda, -\hbar)$ is irreducible.
Since $F$ and $L(\Lambda)$ are isomorphic as $\affsl$-modules, $F$ and $L(\Lambda, -\hbar)$ are isomorphic as $\affY$-modules.

\providecommand{\bysame}{\leavevmode\hbox to3em{\hrulefill}\thinspace}
\providecommand{\MR}{\relax\ifhmode\unskip\space\fi MR }
\providecommand{\MRhref}[2]{%
  \href{http://www.ams.org/mathscinet-getitem?mr=#1}{#2}
}
\providecommand{\href}[2]{#2}


\begin{thebibliography}{FFNR}

\bibitem[FFNR]{MR2827177}
Boris Feigin, Michael Finkelberg, Andrei Negut, and Leonid Rybnikov,
  \emph{Yangians and cohomology rings of {L}aumon spaces}, Selecta Math. (N.S.)
  \textbf{17} (2011), no.~3, 573--607.

\bibitem[FJM]{FJM}
Boris Feigin, Michio Jimbo, and Evgeny Mukhin, \emph{Evaluation modules for
  quantum toroidal $\mathfrak{gl}_n$ algebras}, preprint arXiv:1709.01592.

\bibitem[G1]{MR2199856}
Nicolas Guay, 
\emph{Cherednik algebras and {Y}angians}, Int. Math. Res. Not. (2005), no.~57, 3551--3593.

\bibitem[G2]{MR2323534}
\bysame, 
\emph{Affine {Y}angians and deformed double current algebras in type {A}}, Adv. Math. \textbf{211} (2007), no.~2, 436--484.

\bibitem[GNW]{MR3861718}
Nicolas Guay, Hiraku Nakajima, and Curtis Wendlandt, \emph{Coproduct for
  {Y}angians of affine {K}ac-{M}oody algebras}, Adv. Math. \textbf{338} (2018),
  865--911. 

\bibitem[GRW]{MR4014633}
Nicolas Guay, Vidas Regelskis, and Curtis Wendlandt, \emph{Vertex
  representations for {Y}angians of {K}ac-{M}oody algebras}, J. \'{E}c.
  polytech. Math. \textbf{6} (2019), 665--706. 

\bibitem[K1]{MR3898327}
Ryosuke Kodera, \emph{Affine {Y}angian action on the {F}ock space}, Publ. Res.
  Inst. Math. Sci. \textbf{55} (2019), no.~1, 189--234. 

\bibitem[K2]{MR3869424}
\bysame, \emph{Higher level {F}ock spaces and affine {Y}angian},
  Transform. Groups \textbf{23} (2018), no.~4, 939--962. 

\bibitem[K3]{MR3923494}
\bysame, \emph{Braid group action on affine {Y}angian}, SIGMA Symmetry
  Integrability Geom. Methods Appl. \textbf{15} (2019), 020, 28 pages.

\bibitem[M]{MR1694256}
Kei Miki, 
\emph{Toroidal and level {$0$} {$\mathrm{U}_q'(\widehat{sl_{n+1}})$} actions on {$\mathrm{U}_q(\widehat{gl_{n+1}})$} modules}, J. Math. Phys. \textbf{40} (1999), no.~6, 3191--3210.
  
\bibitem[N]{MR2583334}
Kentaro Nagao, 
\emph{{$K$}-theory of quiver varieties, {$q$}-{F}ock space and
  nonsymmetric {M}acdonald polynomials}, Osaka J. Math. \textbf{46} (2009), no.~3, 877--907.

\bibitem[STU]{MR1603798}
Yoshihisa Saito, Kouichi Takemura, and Denis Uglov, \emph{Toroidal actions on level {$1$} modules of {$U_q(\widehat{\mathfrak{sl}}_n)$}}, Transform. Groups \textbf{3} (1998), no.~1, 75--102.

\bibitem[SV]{MR3150250}
Olivier Schiffmann and Eric Vasserot, \emph{Cherednik algebras, {W}-algebras
  and the equivariant cohomology of the moduli space of instantons on
  {$\bold{A}^2$}}, Publ. Math. Inst. Hautes \'Etudes Sci. \textbf{118} (2013),
  213--342.

\bibitem[TU]{MR1710750}
Kouichi Takemura and Denis Uglov, 
\emph{Representations of the quantum toroidal algebra on highest
  weight modules of the quantum affine algebra of type {$\mathfrak{gl}_N$}},
  Publ. Res. Inst. Math. Sci. \textbf{35} (1999), no.~3, 407--450.

\bibitem[V]{MR1818101}
Michela Varagnolo, 
\emph{Quiver varieties and {Y}angians}, 
Lett. Math. Phys. \textbf{53} (2000), no.~4, 273--283.

\bibitem[VV]{MR1626481}
Michela Varagnolo and Eric Vasserot, \emph{Double-loop algebras and the {F}ock space}, 
Invent. Math. \textbf{133} (1998), no.~1, 133--159.
\end{thebibliography}

\end{document}